\UseRawInputEncoding
\documentclass[12pt]{amsart}
\usepackage{nccmath}

\usepackage{amsmath,amsthm,amssymb}
\usepackage{times}
\usepackage{enumerate}

\newtheorem{theorem}{Theorem}[section]
\newtheorem{corollary}[theorem]{Corollary}
\newtheorem{lemma}[theorem]{Lemma}
\newtheorem{proposition}[theorem]{Proposition}

\newtheorem{conjecture}[theorem]{Conjecture}

\newcommand{\be}{\begin{equation}}
\newcommand{\ee}{\end{equation}}

\newcommand{\lt}{\left}
\newcommand{\rt}{\right}

\newcommand{\goto}{\rightarrow}
\newcommand{\R}{\mathbb{R}}

\newcommand{\e}{\epsilon}
\newcommand{\s}{\sigma}

\newcommand{\bnt}{b_{n, 2}}
\newcommand{\tn}{\tilde\nabla}

\newcommand{\RNum}[1]{\uppercase\expandafter{\romannumeral #1\relax}}

\theoremstyle{definition}
\newtheorem{defin}[theorem]{Definition}

\numberwithin{equation}{section}

\begin{document}
\setlength{\baselineskip}{1.2\baselineskip}

\title[Isoperimetric inequality in De Sitter space]
{An isoperimetric type inequality in De Sitter space}

\author{Ling Xiao}
\address{Department of Mathematics, University of Connecticut, Storrs, Connecticut 06269}
\email{ling.2.xiao@uconn.edu}

\begin{abstract}
In this paper, we prove an optimal isoperimetric inequality for spacelike, compact, star-shaped, and $2$-convex hypersurfaces
in de Sitter space.
\end{abstract}

\maketitle

\section{Introduction}
\label{sec1}
Let $\R^{n+2}_1$ be the $(n+2)$- dimensional Minkowski space, that is, the real vector
space $\R^{n+2}$ endowed with the Lorentz metric
\[\lt<v, w\rt>=-v^0w^0+\sum\limits_{i=1}^{n+1}v^iw^i.\]
The one sheeted hyperboloid
\[\mathbb S^{n+1}_1=\{y\in\R^{n+2}_1: \lt<y, y\rt>=1, y^0>0\}\]
consisting of all unit spacelike vectors and equipped with the induced metric is called
de Sitter space.  It is a geodesically complete simply connected Lorentzian manifold
with constant curvature one. We say a hypersurface $M\subset\mathbb S^{n+1}_1$ is spacelike if its induced metric is Riemannian.

Let $\mathbb S^n$ be the standard round sphere. Then de Sitter space may be parametrised by
$Y: \mathbb S^n\times\R\goto \mathbb S^{n+1}_1$ as follows:
\[Y(r, \xi)=\sinh(r)E_1+\cosh(r)\xi.\]
In this coordinate system, the induced metric is
\[\bar g=-dr^2+\phi^2(r)\s\]
where $\s$ is the standard metric on $\mathbb S^n$ and $\phi=\cosh$. For a hypersurface $M\subset\mathbb S^{n+1}_1,$
we define
\[u=-\lt<\phi\frac{\partial}{\partial r}, \nu\rt>\]
to be the support function, where $\nu$ is the future directed unit normal to $M$ and $\lt<\cdot, \cdot\rt>$ is the inner product with respect to $\bar g$.

In this paper, we prove an optimal isoperimetric inequality for spacelike, compact, star-shaped, and $2$-convex hypersurfaces
in de Sitter space. Before stating our main results, we need the following definition.
\begin{defin}
\label{def1.1}
A $C^2$ regular hypersurface $M\subset\mathbb S^{n+1}_1$
is strictly $k$-convex, if the principal curvature vector of $M$
at $X\in M$ satisfies $\kappa[X]\in\Gamma_k$ for all $X\in M,$ where $\Gamma_k$ is the Garding's cone
\[\Gamma_k=\{\lambda\in\mathbb R^n: \s_j(\lambda)>0, 1\leq j\leq k\}\]
and $\s_j$ is the $j$-th elementary symmetric polynomial. If the principal curvature vector of $M$
at $X\in M$ satisfies $\kappa[X]\in\bar \Gamma_k$ for all $X\in M,$ then we say $M$ is $k$-convex.
\end{defin}
\begin{theorem}
\label{thm1.1}
Let $M_0\subset\mathbb S^{n+1}_1$ be a spacelike, compact, star-shaped, and strictly $2$-convex hypersurface. Then the solution to the following
flow equation
\be\label{1.1}
\left\{\begin{aligned}
X_t&=\lt(u-\bnt\phi'\s_2^{-1/2}\rt)\nu,\\
X_0&=M_0
\end{aligned}
\right.
\ee
exists for all time, where $\bnt=(\s_2(I))^{1/2}=\lt[\frac{n(n-1)}{2}\rt]^{1/2}.$ Moreover, the flow hypersurfaces $M_t$ converge smoothly to a radial coordinate slice as $t\goto\infty.$
\end{theorem}
As a consequence we obtain
\begin{corollary}\label{cor1.1}
Let $M\subset\mathbb S^{n+1}_1$ be a spacelike, compact, star-shaped, and $2$-convex hypersurface. Then there holds
\be\label{1.2}
\int_M\s_2d\mu_g-(n-1)|M|\leq\xi_{2, 0}(|M|)
\ee
with equality is attained if and only if $M$ is a radial coordinate slice. Here, $\xi_{2, 0}$ is the associated monotonically increasing function for radial coordinate slices and $|M|$ denotes the surface area of $M.$
\end{corollary}

\subsection{Background and motivations}The classical Minkowski inequality \cite{Min1903} states that:
For a convex hypesurface $M\subset\mathbb{R}^{n+1}$ we have
\begin{equation}
\label{Min}
\frac{1}{|\mathbb{S}^{n}|}\int_{M}\frac{H}{n}d\mu_g\geq\left(\frac{|M|}{|\mathbb{S}^{n}|}\right)^{\frac{n-1}{n}},
\end{equation}
with equality holds if and only if $M$ is a sphere. Here, $H$ is the mean curvature of $M.$

A natural question,
raised by several authors (see \cite{CW11, Hui09, Tru94} for example), is whether the Minkowski inequality stays true for larger classes of domains than just for convex ones.
By studying weak solutions of the inverse mean curvature flow in $\mathbb R^n,$ Huisken-Ilmanen \cite{HI01, HI08}
showed that the assumption that $M$ is convex can be replaced by the
assumption that $M$ is outward-minimizing. In 2009, by studying a normalized inverse curvature flow, Guan-Li \cite{GL09} proved \eqref{Min} for the case when $M$ is star-shaped and mean convex. Moreover, they also proved Alexandrov-Fenchel inequalities, which is a general form of Minkowski inequality, for star-shaped and $k$-convex hypersurfaces in Euclidean space.

There are analogous of Minkowski and Alexandrov-Fenchel inequalities in space form.
In hyperbolic space, Wang-Xia \cite{WX14} proved the Alexandrov-Fenchel type inequalities for horospherically convex hypersurfaces. Since then, many efforts have been carried out to weaken the condition on the convexity. In particular, for star-shaped and mean convex hypersurfaces, a Minkowski type inequality has been proved in \cite{BGL, SXia19}; for star-shaped and $2$-convex hypersurfaces, an Alexandrov-Fenchel type inequality that involving integral of the scalar curvature has been proved in \cite{BGL, LWX14}. Due to technical reasons there are much less such results in sphere, even with the convexity assumption the Alexandrov-Fenchel type inequalities are still open.
For difficulties in proving the Alexandrov-Fenchel type inequalities in sphere one may refer to the expository paper \cite{CGLS}. Some variants of Alexandrov-Fenchel type inequalities for convex hypersurfaces in sphere can be found in \cite{GP17, WX15}. In de Sitter space, the Alexandrov-Fenchel type inequalities for convex hypersurfaces were deduced through the well-known duality for strictly convex hypersurfaces of hyperbolic/de Sitter space in \cite{AHL20}; while a Minkowski type inequality for spacelike, compact, star-shaped, and mean-convex hypersurfaces was derived in \cite{JS19}.

\subsection{Outline} In Section \ref{sec2}, we give basic notations and establish fundamental equations for geometric quantities in de Sitter space that will be used in later sections. In Section \ref{sec3}, we introduce the flow equation and prove the monotonicity properties for the quermassintegrals along the flow. In Section \ref{sec4},
we will establish a priori estimates for the flow equation \eqref{1.1} and prove that the flow \eqref{1.1} exists for all time. In Section \ref{sec5}, we show the flow converges to a radial coordinate slice. This completes the proof of Theorem \ref{thm1.1} and Corollary \ref{cor1.1}.

\bigskip
\section{Preliminary}
\label{sec2}
In this section, we will collect some formulas and lemmas for $k$-th symmetric functions as well as hypersurfaces in $\mathbb S^{n+1}_1.$
\subsection{Elementary symmetric functions}
\label{sun-section-hession}
For any $k=1, \cdots, n,$ and $\lambda=(\lambda_1, \cdots, \lambda_n)$ the $k$-th elementary symmetric function is defined by
\[\s_k(\lambda)=\sum\limits_{1\leq i_1<i_2<\cdots<i_k\leq n}\lambda_{i_1}\lambda_{i_2}\cdots\lambda_{i_k},\]
and we also define $\s_0=1.$ In this paper we will denote $\s_k(\lambda| i)$ the symmetric function with $\lambda_i=0.$

The following properties are well known.
\begin{lemma}
\label{lem2.1.1}
Let $\lambda=(\lambda_1, \cdots, \lambda_n)\in\mathbb R^n$ and $k=1, \cdots, n,$ then
\[\s_k(\lambda)=\s_k(\lambda|i)+\lambda_i\s_{k-1}(\lambda|i),\,\,\forall 1\leq i\leq n,\]
\[\sum\limits_i\lambda_i\s_{k-1}(\lambda| i)=k\s_k(\lambda),\]
and
\[\sum\limits_i\s_k(\lambda| i)=(n-k)\s_k(\lambda).\]
\end{lemma}

\begin{lemma}
\label{lem2.1.2}
Let $\lambda\in\Gamma_k$ with $\lambda_1\geq\cdots\geq\lambda_k\geq\cdots\geq\lambda_n,$
then we have
\[\s_{k-1}(\lambda| n)\geq\s_{k-1}(\lambda| n-1)\geq\cdots\geq\s_{k-1}(\lambda| 1)>0,\]
\[\lambda_1\geq\cdots\geq\lambda_k>0, \,\,\s_k(\lambda)\leq C_n^k\lambda_1\cdots\lambda_k,\]
\[\sum\limits_i\s_{k-1}(\lambda|i)\lambda_i^2=\s_k\s_1-(k+1)\s_{k+1},\]
where $C_n^k=\frac{n!}{k!(n-k)!}.$
\end{lemma}
The generalized Newton-Maclaurin inequality is as follows, which will be used all the time (see Lemma 2.10 in \cite{SJ05}).
\begin{proposition}
\label{pro2.1.1}
For $\lambda\in\Gamma_k,$ $k>l\geq 0,$ $r>s\geq 0,$ $k\geq r,$ and $l\geq s,$ we have
\[\lt[\frac{\s_k(\lambda)/C_n^k}{\s_l(\lambda)/C_n^l}\rt]^{\frac{1}{k-l}}\leq\lt[\frac{\s_r(\lambda)/C_n^r}{\s_s(\lambda)/C_n^s}\rt]^{\frac{1}{r-s}}.\]
Moreover, the equality holds if and only if $\lambda=c(1, \cdots, 1)$ for some $c > 0.$

\end{proposition}
Let $A$ be a symmetric matrix and $\lambda(A)=(\lambda_1, \cdots, \lambda_n)$ be the eigenvalue vector of $A.$ Let
$F$ be the function defined by
\[F(A)=f(\lambda(A))\]
and denote
\[F^{ij}(A)=\frac{\partial F}{\partial a_{ij}},\,\,F^{pq, rs}=\frac{\partial^2 F}{\partial a_{pq}\partial a_{rs}}.\]
When $A$ is diagonal, we have
\[F^{ij}(A)=f^i(\lambda(A))\delta_{ij},\,\,\mbox{for $f^i=\frac{\partial f}{\partial\lambda_i}.$}\]
Furthermore, we also have
\be
\label{2.1.1}
F^{ij}(A)a_{ij}=\sum\limits_if^i(\lambda(A))\lambda_i,
\ee
\be\label{2.1.2}
F^{ij}(A)a_{ik}a_{kj}=\sum\limits_if^i(\lambda(A))\lambda_i^2.
\ee
In particular, when $F(A)=\s_k(\lambda(A))$ and suppose $A$ is diagonalized at $p_0,$ then at $p_0,$ we have
\be\label{0.3}F^{ij}(A)=\s_k^{ij}(\lambda(A))=\s_{k-1}(\lambda| i)\delta_{ij},\ee
\be\label{0.4}F^{pq, rs}(A)=\s_k^{pq, rs}(\lambda(A))=\left\{\begin{aligned}
&\frac{\partial^2\s_k}{\partial\lambda_p\partial\lambda_r}(\lambda)=\s_{k-2}(\lambda|pr), \,\, &p=q, r=s, p\neq r,\\
&-\frac{\partial^2\s_k}{\partial\lambda_p\partial\lambda_q}(\lambda)=-\s_{k-2}(\lambda|pq), \,\,&p=s, q=r, p\neq q,\\
&0,\,\,&\mbox{otherwise.}
\end{aligned}\right.\ee

In order to prove the long time existence of the flow \eqref{1.1} (see Section \ref{sec4}), we need the following concavity inequality for Hessian operator, which is proved by Siyuan Lu.
\begin{lemma}
\label{lem2.1.0}(Lemma 3.1 of \cite{SYL23})
Let $\lambda=(\lambda_1, \cdots, \lambda_n)\in\Gamma_k$ with $\lambda_1\geq\cdots\geq\lambda_n$ and let $1\leq l<k.$ For any $\epsilon, \delta, \delta_0\in(0, 1),$
there exists a constant $\delta'>0$ depending only on $\epsilon, \delta, \delta_0, n, k$ and $l$ such that if $\lambda_l\geq\delta\lambda_1$ and $\lambda_{l+1}\leq\delta'\lambda_1,$ then we have
\[-\sum\limits_{p\neq q}\frac{\s_k^{pp, qq}\xi_p\xi_q}{\s_k}+\frac{(\sum_i\s_k^{ii}\xi_i)^2}{\s_k^2}
\geq(1-\epsilon)\frac{\xi_1^2}{\lambda_1^2}-\delta_0\sum\limits_{i>l}\frac{\s_k^{ii}\xi_i^2}{\lambda_1\s_k},\]
where $\xi=(\xi_1, \cdots, \xi_n)$ is an arbitrary vector in $\mathbb R^n.$
\end{lemma}
Note that from the proof in \cite{SYL23}, we can see that for fixed $\delta, \delta_0\in(0, 1),$ $\delta'=\delta'(\e, \delta, \delta_0, n, k)=O(\e)>0$ is a small constant.

\subsection{Star-shaped graph in $\mathbb S^{n+1}_1$}
\label{sub-section-star}
Let $M=\{(\rho(\xi), \xi): \xi\in\mathbb S^n\},$ we will use $\tn$ to denote the standard covariant derivative for the metric $\s$
on $\mathbb S^n.$ Then the tangent space of the hypersurface at a point $Y\in M$ is spanned by
\[Y_j=\rho_j\partial_r+\partial_j,\,\,\mbox{where $\partial_j:=\tn_{\xi_j}$},\]
and the induced metric on $M$ is given by
\[g_{ij}=\lt<Y_i, Y_j\rt>=-\rho_i\rho_j+\cosh^2(\rho)\sigma_{ij}.\]
$M$ is spacelike if $(g_{ij})$ is positive-definite.
A unit normal vector $\nu$ to $M$ can be obtained by solving the equation
$\lt<Y_i, \nu\rt>=0$ for $\forall 1\leq i\leq n.$ Thus we have
\[\nu=\frac{(\cosh\rho, \tn\rho/\cosh\rho)}{\sqrt{\cosh^2(\rho)-|\tn\rho|^2}},\]
here $|\tn\rho|^2=\sigma^{ij}\rho_i\rho_j$ and $(\sigma^{ij})$ is the inverse of $(\sigma_{ij}).$
In the following, for our convenience we will denote $w:=\sqrt{\cosh^2(\rho)-|\tn\rho|^2},$ then the support function
\be\label{supp}u:=-\lt<\cosh \rho\frac{\partial}{\partial\rho}, \nu\rt>=\frac{\cosh^2(\rho)}{w}.\ee
Moreover, by some routine calculations (for details see \cite{BKL}) we get
\be\label{star.1}
g^{ij}=\frac{1}{\cosh^2(\rho)}\lt(\s^{ij}+\frac{\rho^i\rho^j}{w^2}\rt),
\ee
and
\be\label{star.2}
h_{ij}=\frac{\cosh\rho}{w}\lt(\tn_{ij}\rho-2\rho_i\rho_j\tanh\rho+\sinh\rho\cosh\rho\sigma_{ij}\rt),
\ee
where $(g^{ij})$ is the inverse of $(g_{ij}),$ $\rho^i=\s^{il}\rho_l,$
and $h_{ij}$ is the second fundamental form of $M.$

\subsection{Hypersurfaces in $\mathbb S^{n+1}_1$}
\label{sub-section-hypersurface}
In this paper, we will define
\be\label{4.1}\Phi=-\int^r_0\cosh sds=-\sinh r,\ee
and
\be\label{4.1.1}
V=\cosh r\frac{\partial}{\partial r}.
\ee
We note that $\Phi=-\phi'.$ Now, let $M\in\mathbb S^{n+1}_1$ be a spacelike hypersurface with induced metric $g$. We will use
$\nabla$ to denote the covariant derivative with respect to $g.$
Then the following fundamental equations are well known:
\[\begin{aligned}
&\nabla_{\tau_i}\tau_j=h_{ij}\nu\,\,&\mbox{Gauss formula},\\
&\nabla_{\tau_i}\nu=h_i^k\tau_k\,\,&\mbox{Weingarten equation},\\
&h_{ijk}=h_{ikj}\,\,&\mbox{Codazzi equation.}
\end{aligned}\]

Following the proof of Lemma 2.2 in \cite{GL15} (see also equation (2.7) in \cite{JS19}) we have
\begin{lemma}
\label{lem4.1}
Let $M\subset\mathbb S^{n+1}_1$ be a spacelike, compact, connected hypersurfaces with induced metric $g$. Let $\Phi$ be defined as in \eqref{4.1}. Then
$\Phi\mid_{M}$ satisfies,
\be\label{4.2}
\nabla_{ij}\Phi=\phi'(\rho)g_{ij}-h_{ij}u,
\ee
where $\nabla$ is the covariant derivative with respect to $g,$ $h_{ij}$ is the second fundamental form of $M,$ and $u=-\lt<V, \nu\rt>$ is the support function of $M.$
\end{lemma}

Next, following the proof of Lemma 2.6 in \cite{GL15}, we derive the gradient and hessian of the support function $u$ under the induced metric $g$ on $M.$
\begin{lemma}
\label{lem4.2}
The support function $u$ satisfies
\be\label{4.3}
\nabla_iu=-h_i^k\nabla_k\Phi,
\ee
\be\label{4.4}
\nabla_{ij}u=-g^{kl}\nabla_kh_{ij}\nabla_l\Phi-\phi'h_{ij}+uh^k_ih_{kj},
\ee
where $h_i^k=g^{kl}h_{li}.$
\end{lemma}

\bigskip
\section{Curvature flow and monotonicity formula}
\label{sec3}
In this paper, we consider hypersurface flows related to the quermassintegrals. Similar to \cite{CGLS}, let $M=\partial\Omega,$ set
\begin{equation}
\begin{aligned}
\mathcal{A}_{-1}&=\text{Vol}(\Omega),\,\,\mathcal A_0=\int_Md\mu_g\\
\mathcal A_1&=\int_M\s_1d\mu_g-n\text{Vol}(\Omega)\\
\mathcal A_m&=\int_M\s_md\mu_g-\frac{n-m+1}{m-1}\mathcal A_{m-2},
\end{aligned}
\end{equation}
where $2\leq m\leq n,$ $g$ is the induced metric on $M,$ and $d\mu_g$ is the associated volume element.
Let $M_t$ be a smooth family of spacelike, compact, connected hypersurfaces in $\mathbb S^{n+1}_1$ evolving along the flow
\be\label{3.0}
X_t=S\nu.
\ee
From \cite{Ger} we get
\be\label{v1.1}
\partial_tg_{ij}=2Sh_{ij}
\ee
and
\be\label{v1.2}
\partial_t\nu=\nabla S.
\ee
Moreover, by Lemma 3.1 of \cite{JS19} we have
\be\label{evolution-curvature}
\partial_th_i^j=\nabla^j\nabla_i S-Sh^k_ih^j_k+S\delta_i^j.
\ee
By Lemma 3.2 of \cite{JS19} we also have
\be\label{3.1}
\begin{aligned}
\partial_t\mathcal A_{-1}&=\int_{M_t}Sd\mu_g,\\
\partial_t\mathcal A_{0}&=\int_{M_t}\s_1Sd\mu_g.
\end{aligned}
\ee
Combining \eqref{evolution-curvature}, \eqref{3.1} with an induction argument, we derive for $0\leq l\leq n-1$
\be\label{3.2}
\partial_t\mathcal A_l=(l+1)\int_{M_t}S\s_{l+1}d\mu_g.
\ee
\subsection{A specific flow equation.} In order to obtain an isoperimetric type inequality, we will consider the following curvature flow
\be\label{general-flow-equation}
X_t=\lt(u-b_{n, k}\phi'\s_k^{-1/k}\rt)\nu,
\ee
where $b_{n, k}=\lt(C_n^k\rt)^{1/k}$ and $1\leq k\leq n.$ Then from now on, our normal velocity $S=u-b_{n, k}\phi'\s_k^{-1/k}.$

Recall the Hsiung-Minkowski identities (see (2.8) of \cite{JS19} or (1.4) of \cite{CGLS})
\be\label{3.3}
\int_Mu\s_{m+1}d\mu_g=C_{n, m}\int_M\phi'\s_md\mu_g,\,\,0\leq m\leq n-1
\ee
for $C_{n, m}=\frac{n-m}{m+1}=\frac{C_n^{m+1}}{C_n^m}$ we obtain the following lemma.

\begin{lemma}
\label{lem3.1} Let $M_t$ be a smooth family of spacelike, compact, connected, strictly $k$-convex hypersurfaces in $\mathbb S^{n+1}_1$ evolving along the flow \eqref{general-flow-equation}.
Then the surface area $\mathcal A_0$ is non-increasing and the quantity
\[\mathcal A_{k}(M_t)=\left\{\begin{aligned}
&\int_{M_t}\s_1d\mu_g-n\text{Vol}(\Omega_t),\,\,k=1\\
&\int_{M_t}\s_kd\mu_g-\frac{n-k+1}{k-1}\mathcal A_{k-2}(M_t),\,\,2\leq k\leq n-1
\end{aligned}\right.\]
is non-decreasing. Moreover, $\mathcal A_k$ is strictly increasing unless $M_t$ is totally umbilic.
\end{lemma}
\begin{proof}
In view of \eqref{3.2} and \eqref{3.3} we get, along the flow \eqref{general-flow-equation}
\be\label{3.4}
\begin{aligned}
\partial_t \mathcal A_0&=\int_{M_t}u\s_1-b_{n, k}\phi'\s_k^{-1/k}\s_1d\mu_g\\
&=\int_{M_t}n\phi'-b_{n, k}\phi'\s_k^{-1/k}\s_1d\mu_g.
\end{aligned}
\ee
By the Newton-Maclaurin inequality (see Proposition \ref{pro2.1.1}) we know that when $M_t$ is strictly $k$-convex,
\[\frac{\s_1}{n}\geq\frac{\s_k^{1/k}}{b_{n, k}}.\]
Thus,we conclude that $\partial_t\mathcal A_0\leq 0.$

Similarly, we can also obtain for $1\leq k\leq n-1$
\be\label{3.5}
\begin{aligned}
\partial_t \mathcal A_k&=(k+1)\int_{M_t}\s_{k+1}(u-b_{n, k}\phi'\s_k^{-1/k})d\mu_g\\
&=(k+1)\int_{M_t}C_{n, k}\phi'\s_k-b_{n, k}\phi'\s_k^{-1/k}\s_{k+1} d\mu_g.
\end{aligned}
\ee
It's easy to see that at the point where $\s_{k+1}\leq 0$ we have $C_{n, k}\phi'\s_k-b_{n, k}\phi'\s_k^{-1/k}\s_{k+1}>0;$ while at the point where
$\s_{k+1}>0,$ applying Newton-Maclaurin inequality we still have $C_{n, k}\phi'\s_{k}-b_{n, k}\phi'\s_k^{-1/k}\s_{k+1}\geq0.$ Moreover, the equality holds if and only if at this
point the principal curvature vector of $M_t$ is $\kappa=c(1, \cdots, 1)$ for some $c>0.$ Therefore, the lemma is proved.
\end{proof}
We want to point out that a straightforward calculation yields $\partial_t\mathcal A_n\equiv 0.$
Hence, if one can prove the flow \eqref{general-flow-equation} moves an arbitrary $k$-convex hypersurface to a round sphere, then the following conjecture would turn into a theorem:
\begin{conjecture}
Let $M\subset\mathbb S^{n+1}_1$ be a spacelike, compact, star-shaped, and $k$-convex hypersurface. Then there holds
\be\label{conj1}
\mathcal A_k\leq\xi_{k, 0}(\mathcal A_0),\,\,1\leq k\leq n
\ee
with equality holds if and only if $M$ is a radial coordinate slice. Here, $\xi_{k, 0}$ is the associated monotonically increasing function for radial coordinate slices.
\end{conjecture}

Note that when $k=1$ the above conjecture has been proved in \cite{JS19}. In this paper, we solve the case when $k=2.$

\bigskip
\section{Long time existence of \eqref{1.1}}
\label{sec4}

In this section, we will establish a priori estimates for the flow equation \eqref{1.1} and prove the long time existence theory.
For greater generality, we will start with the study of \eqref{general-flow-equation} instead. In the rest of this section, we will assume
the initial hypersurface $M_0$ is spacelike, compact, star-shaped, and strictly $k$-convex. Then by the short time existence theorem we know there exists $T^*>0$ such that the flow \eqref{general-flow-equation} has a unique solution $M_t$ for $t\in[0, T^*).$ Moreover, the flow hypersurface $M_t$ is also spacelike, compact, star-shaped, and strictly $k$-convex.

\subsection{Estimates up to first order}
\label{sub-section 4.1}
In this subsection, we will derive the $C^0$ and $C^1$ estimates for the solution of \eqref{general-flow-equation}.
\begin{lemma}
\label{lem-c0}
Along the flow \eqref{general-flow-equation} there holds for all $(\xi, t)\in\mathbb S^n\times(0, T^*)$ we have
\[\min\limits_{\mathbb S^n}\rho(\cdot, 0)\leq\rho(\xi, t)\leq\max\limits_{\mathbb S^n}\rho(\cdot, 0).\]
\end{lemma}
\begin{proof}
The proof is the same as the one in \cite{JS19}, for completeness, we include it here.

The radial function $\rho$ satisfies
\[\rho_t=\lt(u-b_{n,k}\phi'\s_k^{-1/k}\rt)\frac{\cosh\rho}{w}.\]
At the critical point of $\rho,$ we get
\[\tn\rho=0\,\,\mbox{and $w=\cosh\rho=u.$}\]
In view of \eqref{star.1} and \eqref{star.2} we can see that at the the critical point,
\[h^i_j=g^{ik}h_{kj}=\frac{1}{\cosh^2(\rho)}\lt(\tn_{ij}\rho+\sinh\rho\cosh\rho\delta_{ij}\rt).\]
Therefore, we obtain that at the critical point
\[\rho_t=\cosh\rho-\frac{b_{n,k}\sinh\rho}{\s_k^{1/k}\lt(\frac{\tn_{ij}\rho}{\cosh^2(\rho)}+\frac{\sinh\rho}{\cosh\rho}\delta_{ij}\rt)}.\]
Note that at the spacial maximal points of $\rho$ we have $(\tn_{ij}\rho)\leq 0$, which implies that
\[\s_k^{1/k}\lt(\frac{\tn_{ij}\rho}{\cosh^2(\rho)}+\tanh\rho\delta_{ij}\rt)\leq b_{n,k}\tanh\rho.\]
Thus we have $\max\rho$ is non-increasing. Similarly, we can show that $\min\rho$ is non-decreasing.
\end{proof}

Define $L:=\partial_t-b_{n,k}\phi'F^{-2}F^{ij}\nabla^j\nabla_i$ for $F=\s_k^{1/k}$ and $F^{ij}=\frac{\partial F}{\partial h^i_j},$
denote $$\hat L: =L+\lt<V, \nabla\cdot\rt>,$$ we will use Lemma \ref{lem4.1} and Lemma \ref{lem4.2} to derive the evolution equations for $\Phi$ and $u.$
\begin{lemma}
\label{lem4.3}
Along the flow \eqref{general-flow-equation}, $\Phi$ and $u$ evolve as follows
\be\label{flow-Phi}
\hat L\Phi=2b_{n,k} F^{-1}\phi'u-b_{n,k}(\phi')^2F^{-2}\sum f^i-\cosh^2(\rho),
\ee
and
\be\label{flow-supp}
\hat Lu=\phi'u\lt(1-b_{n,k}F^{-2}\sum f^i\kappa_i^2\rt)-b_{n, k} F^{-1}|\nabla\Phi|^2,
\ee
where $F(A)=f(\kappa[A])$ and $\sum f^i=\sum F^{ii}.$
\end{lemma}
\begin{proof}
Since the values of $\hat L\Phi$ and $\hat Lu$ are independent of the choice of coordinates, we may always choose an orthonormal frame $\{\tau_1, \cdots, \tau_n\}$ such that
$(h_{ij})$ is diagonalized. Then at the point of consideration, we have $F^{ij}=f^i\delta_{ij}$ and $h_{ij}=\kappa_i\delta_{ij}.$
In view of Lemma \ref{lem4.1} we get
\begin{align*}
\hat L\Phi&=\partial_t\Phi-b_{n,k}\phi'F^{-2}F^{ii}\nabla_i\nabla_i\Phi+\lt<V, \nabla\Phi\rt>\\
&=\lt<V, S\nu\rt>-b_{n,k}\phi'F^{-2}F^{ii}(\phi'\delta_{ii}-h_{ii}u)+|\nabla\Phi|^2\\
&=-(u-b_{n,k} F^{-1}\phi')u-b_{n,k}(\phi')^2F^{-2}\sum f^i+b_{n,k}\phi'F^{-1}u+|\nabla\Phi|^2\\
&=2b_{n,k} F^{-1}\phi'u-b_{n,k}(\phi')^2F^{-2}\sum f^i-\cosh^2(\rho).
\end{align*}
Here, we have used $\sum f^i \kappa_i=f$ and
$$\lt<V, \nabla\Phi\rt>=\sum\lt<V, \lt<V, \tau_i\rt>\tau_i\rt>=|V|^2+u^2=u^2-\cosh^2(\rho)=|\nabla\Phi|^2.$$
Similarly, by Lemma \ref{lem4.2} and \eqref{v1.2} we have
\begin{align*}
\hat Lu&=\partial_tu-b_{n,k}\phi'F^{-2}F^{ii}\nabla_i\nabla_iu+\lt<V, \nabla u\rt>\\
&=\phi'S-\lt<V, \nabla S\rt>+\lt<V, \nabla u\rt>-b_{n,k}\phi'F^{-2}F^{ii}(-h_{iik}\nabla_k\Phi-\phi'h_{ii}+uh^l_ih_{li})\\
&=\phi'S+b_{n,k}\lt<V, (\nabla\phi')F^{-1}\rt>+b_{n,k}\lt<V, \phi'\nabla F^{-1}\rt>+b_{n,k}\phi'F^{-2}\lt<\nabla F, V\rt>\\
&+b_{n,k}(\phi')^2F^{-1}-ub_{n,k}\phi'F^{-2}F^{ii}h^l_ih_{li}\\
&=\phi'u-b_{n,k} F^{-1}|\nabla\Phi|^2-ub_{n,k}\phi'F^{-2}\sum f^i\kappa_i^2\\
&=\phi'u\lt(1-b_{n,k} F^{-2}\sum f^i\kappa_i^2\rt)-b_{n,k} F^{-1}|\nabla\Phi|^2.
\end{align*}
\end{proof}

From Lemma \ref{lem4.3} we obtain the $C^1$ estimate.

\begin{lemma}
\label{lem-c1}
Along the flow \eqref{general-flow-equation} there holds for all $(\xi, t)\in\mathbb S^n\times(0, T^*)$ we have
\[u(\xi, t)\leq\max\limits_{\mathbb S^n}u(\cdot, 0).\]
\end{lemma}
\begin{proof}
Let $\kappa=(\kappa_1, \cdots, \kappa_n)$ be the principal curvature vector of $M_t,$ then $\kappa\in\Gamma_k.$
In view of Lemma \ref{lem2.1.2} and Proposition \ref{pro2.1.1} we get
\begin{align*}
&\sum f^i\kappa_i^2=\frac{1}{k}\s_k^{1/k-1}\s_{k-1}(\kappa|i)\kappa_i^2\\
&=\frac{1}{k}\s_k^{1/k-1}[\s_k\s_1-(k+1)\s_{k+1}]\\
&\geq\frac{1}{k}\s_k^{1/k}\lt[\frac{n}{b_{n, k}}\s_k^{1/k}-(k+1)C_{n, k}\frac{\s_k^{1/k}}{b_{n, k}}\rt].
\end{align*}
Therefore, we have \be\label{0.1}\sum f^i\kappa_i^2\geq\frac{f^2}{b_{n, k}}.\ee
Combining with equation \eqref{flow-supp} we obtain, along the flow \eqref{general-flow-equation}  $\hat Lu\leq 0$. Then the lemma follows from the standard maximum principle.
\end{proof}

Recall equation \eqref{supp}, we can see that Lemma \ref{lem-c1} implies that along the flow \eqref{general-flow-equation}, $\frac{|\tn\rho|^2}{\cosh^2(\rho)}$ is uniformly bounded away from 1.

\subsection{Uniform bounds of $F$} In this subsection, we will prove that $F$ is uniformly bounded from below along the flow
\eqref{general-flow-equation}. However, due to technical reasons, to obtain the uniform upper bound of $F$ we have to restrict ourselves to the case when $k=2,$ that is, the flow \eqref{1.1}.

\label{sub-section 4.2}
\begin{lemma}
\label{lem-F lower bound}
Along the flow \eqref{general-flow-equation} there holds for all $(\xi, t)\in\mathbb S^n\times(0, T^*)$ we have
\[F(\xi, t)\geq\min\limits_{\mathbb S^n}F(\cdot, 0).\]
\end{lemma}
\begin{proof}
At the critical point of $F$ we may choose an orthonormal frame $\{\tau_1, \cdots, \tau_n\}$ such that
$h_{ij}=\kappa_i\delta_{ij}$ is diagonalized.
By virtue of \eqref{evolution-curvature} we have
\begin{align*}
\partial_tF&=F^{ii}\lt\{u_{ii}+b_{n, k} F^{-1}\Phi_{ii}+2b_{n,k} F^{-1}F_i\lt(\frac{\phi'}{F}\rt)_i+b_{n,k}\phi'F^{-2}\nabla_{ii}F\rt\}\\
&-(u-b_{n,k}\phi'F^{-1})\sum f^i\kappa_i^2+(u-b_{n,k}\phi'F^{-1})\sum f^i.
\end{align*}
Applying Lemma \ref{lem4.1} and Lemma \ref{lem4.2} we obtain
\be\label{flow-F}
\begin{aligned}
\partial_tF&=b_{n,k}\phi'F^{-2}F^{ii}\nabla_{ii}F+F^{ii}(-h_{iil}\nabla_l\Phi-\phi'h_{ii}+uh^l_ih_{li})+b_{n,k} F^{-1}F^{ii}(\phi'\delta_{ii}-h_{ii}u)\\
&+2b_{n,k} F^{-1}F^{ii}F_i\lt(\frac{\phi'}{F}\rt)_i-(u-b_{n,k}\phi'F^{-1})\sum f^i\kappa_i^2\\
&+(u-b_{n, k}\phi'F^{-1})\sum f^i.
\end{aligned}
\ee
Let $F_{\min}(t)=\min\limits_{\xi\in\mathbb S^n}F(\xi, t),$ then $F_{\min}$ satisfies
\begin{align*}
\frac{d}{dt}F_{\min}&\geq-\phi'F+u\sum f^i\kappa_i^2+b_{n,k} F^{-1}\phi'\sum f^i-b_{n,k} u\\
&-u\sum f^i\kappa_i^2+b_{n,k}\phi'F^{-1}\sum f^i\kappa_i^2+u\sum f^i-b_{n,k}\phi'F^{-1}\sum f^i\\
&=u(\sum f^i-b_{n,k})+\phi'F^{-1}\lt(b_{n,k}\sum f^i\kappa_i^2-F^2\rt)
\end{align*}
By the concavity of $F$ we get
\be\label{0.2}\sum f^i\geq f(1, \cdots, 1)=b_{n,k}.\ee
In conjunction with inequality \eqref{0.1} we conclude
\[\frac{d}{dt}F_{\min}\geq 0,\]
which yields this lemma.
\end{proof}

In the proofs of Lemma \ref{lem-F upper bound} and Lemma \ref{lem-c2} below, we will explicitly use the fact that  $k=2.$ Therefore, from now on, we will restrict ourselves to the locally constrained inverse scalar curvature flow
\be\label{flow-equation}
\left\{\begin{aligned}
X_t&=\lt(u-\bnt\phi'\s_2^{-1/2}\rt)\nu,\\
X_0&=M_0.
\end{aligned}\right.
\ee
\begin{lemma}
\label{lem-F upper bound}
Along the flow \eqref{flow-equation} there holds for all $(\xi, t)\in\mathbb S^n\times(0, T^*)$ we have
\[F(\xi, t)\leq C,\]
where $C>0$ is a constant depending on $M_0, n, \rho,$ and $u$.
\end{lemma}
\begin{proof}
Recall the flow equation \eqref{flow-Phi} of $\Phi$ we have
\begin{align*}
\hat L\Phi&=2\bnt F^{-1}\phi'u-\bnt(\phi')^2F^{-2}\sum f^i-\cosh^2(\rho)\\
&\leq-\lt(u-\frac{\bnt\phi'}{F}\rt)^2+u^2-\cosh^2(\rho),
\end{align*}
where we have used $\sum f^i\geq\bnt.$
Moreover, by Lemma \ref{lem-c1} we get
\[u^2-\cosh^2(\rho)=\frac{\cosh^2(\rho)}{1-|\tn\rho|^2/\cosh^2(\rho)}-\cosh^2(\rho)\leq\beta_0 u^2,\]
for some $0<\beta_0=\beta_0(|\tn\rho|/\cosh\rho)<1.$
Therefore, we can see that there exists some uniform constant $0<\beta_1=\beta_1(n, \beta_0)<1$ such that whenever $F\geq C_0(n, \rho, u)$
\be\label{flow-Phi*}\hat L\Phi\leq-\beta_1 u^2.\ee
Here, $C_0(n, \rho, u)$ is a large constant that only depends on $n,$ $\rho,$ and $u.$

Now, consider $\Psi=\log F+\lambda u+\alpha\Phi,$ where $\lambda, \alpha>0$ to be determined. Assume $\Psi$ achieves its maximum at an interior point
$X_0\in M_{t_0}$ for some $t_0\in(0, T^*).$ Then at $X_0$ we can choose an orthonormal frame $\{\tau_1, \cdots, \tau_n\}$ such that $h_{ij}=\kappa_i\delta_{ij}.$
We have, at $X_0$
\be\label{4.2.1}
\Psi_i=\frac{F_i}{F}+\lambda u_i+\alpha\Phi_i=0
\ee
and
\be\label{4.2.2}
\begin{aligned}
0&\leq\hat L\Psi=\frac{\hat LF}{F}+\bnt\phi'F^{-2}F^{ii}\lt(\frac{F_i}{F}\rt)^2+\lambda\hat Lu+\alpha\hat L\Phi\\
&\leq\frac{1}{F}
\lt[u(\sum f^i-\bnt)+\phi'F^{-1}(\bnt\sum f^i\kappa_i^2-F^2)\right.\\
&\left.+2\bnt F^{-1}F^{ii}\lt(\frac{\phi'_iF_i}{F}-\frac{\phi'F_i^2}{F^2}\rt)\rt]+\bnt\phi'F^{-4}F^{ii}F_i^2\\
&+\lambda\lt(\phi'u-\bnt\phi'uF^{-2}\sum f^i\kappa_i^2-\bnt F^{-1}|\nabla\Phi|^2\rt)+\alpha\hat L\Phi,
\end{aligned}
\ee
where we have used equations \eqref{flow-supp} and \eqref{flow-F}.
We can see that \eqref{4.2.2} implies
\be\label{4.2.3}
\begin{aligned}
0&\leq\frac{u}{F}\sum f^i+\phi'F^{-2}\bnt\sum f^i\kappa_i^2\lt(1-\lambda u\rt)\\
&+2\bnt F^{-3}F^{ii}\phi'_iF_i-\bnt\phi'F^{-4}F^{ii}F_i^2+\lambda\phi'u+\alpha\hat L\Phi.
\end{aligned}
\ee

Since
\[\begin{aligned}
2\bnt F^{-3}F^{ii}\phi'_iF_i&\leq\bnt\phi'F^{-4}F^{ii}F_i^2+\frac{\bnt}{\phi'}F^{-2}F^{ii}(\phi'_i)^2\\
&\leq\bnt\phi'F^{-4}F^{ii}F_i^2+\frac{\bnt}{\phi'}F^{-2}|\nabla\Phi|^2\sum f^i,\end{aligned}\]
if at the critical point $F\geq C_1=C_1(u, n, \rho)$ very large, applying \eqref{flow-Phi*} we can see that \eqref{4.2.3} becomes
\be\label{4.2.4}
\begin{aligned}
0&\leq\hat L\Psi\leq\sum f^i+\phi'F^{-2}\bnt\sum f^i\kappa_i^2\lt(1-\lambda u\rt)\\
&+\lambda\phi'u-\alpha\beta_1u^2.
\end{aligned}
\ee
From equation \eqref{0.3}, Lemma \ref{lem2.1.1}, Lemma \ref{lem2.1.2}, and Proposition \ref{pro2.1.1} we obtain
\be\label{0.5}\sum f^i=\frac{1}{2}F^{-1}(n-1)\s_1\ee
and
\be\label{0.6}\sum f^i\kappa_i^2=\frac{1}{2}F\s_1-\frac{3}{2}F^{-1}\s_3\geq\frac{1}{n}\s_1 F.\ee
Thus we may choose $\lambda=\lambda(\rho, n, u)>0$ such that
\[\sum f^i+\phi'F^{-2}\bnt\sum f^i\kappa_i^2\lt(1-\lambda u\rt)\leq-\sum f^i.\]
After we fix the value of $\lambda,$ we can choose $\alpha=\alpha(\rho, u, \lambda, \beta_1)>0$ such that
\[\lambda\phi'u-\alpha\beta_1u^2<0.\]
We conclude that if $F>C_1$ large at $X_0$ we would have the right hand side of \eqref{4.2.4} is negative. This leads to a contradiction.
Therefore, $\Psi$ is uniformly bounded on $[0, T^*).$
\end{proof}

\subsection{Uniform bounds for principal curvatures}
\label{sub-section 4.3}
In this subsection, we will show that the principal curvatures of $M_t$ remain bounded along
the flow \eqref{flow-equation}. More precisely, we will prove the following lemma
\begin{lemma}
\label{lem-c2}
Let $M_t$ be the solution of \eqref{flow-equation} on $[0, T^*),$ then there exists a constant $C$ depending on
$M_0, n, \rho, u,$ and $F$ such that
\be\label{c2bound}
\lt|\kappa[M_t]\rt|\leq C.
\ee
\end{lemma}
\begin{proof}
Let us consider the test function
\[G=\log\kappa_1+\lambda\Phi,\]
where $\kappa_1$ is the largest principal curvature and $\lambda>0$ is a large constant to be determined. Assume $G$ achieves its maximum at an interior point $P_0\in M_{t_0}$ for some $t_0\in (0, T^*).$
Then at this point, we can choose an orthonormal frame such that $h_{ij}=\kappa_i\delta_{ij}$ is diagonalized. Without loss of generality, we may assume
$\kappa_1$ has multiplicity $m,$ i.e.,
\[\kappa_1=\cdots=\kappa_m>\kappa_{m+1}\geq\cdots\geq\kappa_n.\]
By Lemma 5 in \cite{BCD} (see also \cite{SYL23}) we have at $P_0,$
\be\label{4.3.1}
\delta_{kl}\kappa_{1i}=h_{kli},\mbox{for $1\leq k, l\leq m,$}
\ee
and
\be\label{4.3.2}
\kappa_{1ii}\geq h_{11ii}+2\sum\limits_{p>m}\frac{h^2_{1pi}}{\kappa_1-\kappa_p}
\ee
in the viscosity sense. We want to point out that \eqref{4.3.1} yields that when $m\geq 2,$ $h_{11i}=0$ for $1<i\leq m.$

We will start with computing the evolution equation of $\kappa_1$ at $P_0.$
Recall \eqref{evolution-curvature} we have
\begin{align*}
\partial_th_i^i&=\nabla_{ii}S-S\kappa_i^2+S\\
&=u_{ii}-\bnt\phi'_{ii}F^{-1}+2\bnt\phi'_iF^{-2}F_i-2\bnt\phi'F^{-3}F^2_i\\
&+\bnt\phi'F^{-2}\lt(F^{kk}h_{kkii}+F^{pq, rs}h_{pqi}h_{rsi}\rt)-S\kappa_i^2+S.
\end{align*}
Combining with the following commuting formula in de Sitter space (for details see page 10 of \cite{BKL})
\[h_{kkii}=h_{iikk}+\kappa_i^2\kappa_k-\kappa_i\delta_{kk}-\kappa_i\kappa_k^2+\kappa_k\delta_{ii},\] we deduce
\begin{align*}
\partial_th^i_i&=(-h_{iik}\nabla_k\Phi-\phi'\kappa_i+u\kappa_i^2)-\bnt F^{-1}(-\phi'\delta_{ii}+\kappa_iu)\\
&+2\bnt F^{-2}\phi'_iF_i-2\bnt\phi'F^{-3}F^2_i\\
&+\bnt\phi' F^{-2}F^{kk}(h_{iikk}+\kappa_i^2\kappa_k-\kappa_i\delta_{kk}-\kappa_i\kappa_k^2+\kappa_k\delta_{ii})\\
&-(u-\bnt\phi'F^{-1})\kappa_i^2+(u-\bnt\phi'F^{-1})\\
&+\bnt\phi'F^{-2}F^{pq, rs}h_{pqi}h_{rsi}.
\end{align*}
Therefore, we obtain
\be\label{flow-curvature}
\begin{aligned}
\hat Lh^i_i&=\kappa_i\lt(-\phi'-\bnt F^{-1}u-\bnt\phi'F^{-2}\sum f^k-\bnt\phi'F^{-2}\sum f^k\kappa_k^2\rt)\\
&+\lt(\bnt\phi'F^{-1}+u\rt)+2\bnt\phi'F^{-1}\kappa_i^2\\
&+\lt(2\bnt F^{-2}\phi'_iF_i-2\bnt\phi'F^{-3}F^2_i+\bnt\phi'F^{-2}F^{pq, rs}h_{pqi}h_{rsi}\rt).
\end{aligned}
\ee
From \eqref{4.3.2} we get, at $P_0$
\[\hat L\kappa_1\leq\hat Lh^1_1-2\bnt\phi'F^{-2}\sum\limits_i\sum\limits_{p>m}F^{ii}\frac{h^2_{1pi}}{\kappa_1-\kappa_p}.\]

By our assumption we have, at $P_0$
\be\label{4.3.3}
G_i=\frac{\kappa_{1i}}{\kappa_1}+\lambda\Phi_i=0.
\ee
Moreover, in view of Lemma \ref{lem4.3} and equation \eqref{flow-curvature} we derive
\be\label{4.3.4}
\begin{aligned}
0&\leq\hat LG=\frac{\hat L\kappa_1}{\kappa_1}+\frac{\bnt\phi'F^{-2}F^{ii}}{\kappa_1^2}h_{11i}^2+\lambda\hat L\Phi\\
&\leq\frac{1}{\kappa_1}\bigg[\kappa_1\lt(-\phi'-\bnt F^{-1}u-\bnt\phi'F^{-2}\sum f^k-\bnt\phi'F^{-2}\sum f^k\kappa_k^2\rt)\\
&+(\bnt\phi'F^{-1}+u)+2\bnt\phi'F^{-1}\kappa_1^2
+\lt(2\bnt F^{-2}\phi'_1F_1-2\bnt\phi'F^{-3}F_1^2\right.\\
&\left.+\bnt\phi'F^{-2}F^{pq, rs}h_{pq1}h_{rs1}\rt)
-2\bnt\phi'F^{-2}\sum\limits_i\sum\limits_{p>m}F^{ii}\frac{h^2_{1pi}}{\kappa_1-\kappa_p}\bigg]\\
&+\frac{\bnt\phi'F^{-2}}{\kappa_1^2}F^{ii}h^2_{11i}+\lambda\lt(2\bnt F^{-1}\phi'u-\bnt(\phi')^2F^{-2}\sum f^k-\cosh^2(\rho)\rt).
\end{aligned}
\ee
Since
\[2\bnt F^{-2}\phi'_1F_1\leq\frac{1}{2}\bnt\phi'F^{-3}F_1^2+2\frac{\bnt F^{-1}(\phi'_1)^2}{\phi'}\]
\eqref{4.3.4} becomes
\be\label{4.3.5}
\begin{aligned}
0&\leq\hat LG\leq\lambda C_1+\lt(-\bnt\phi'F^{-2}\sum f^k-\bnt\phi'F^{-2}\sum f^k\kappa_k^2\rt)\\
&-\frac{3}{2\kappa_1}\bnt\phi'F^{-3}F_1^2+\frac{\bnt}{\kappa_1}\phi'F^{-2}F^{pq, rs}h_{pq1}h_{rs1}\\
&+2\bnt\phi'F^{-1}\kappa_1-\frac{2\bnt}{\kappa_1}\phi'F^{-2}\sum\limits_i\sum\limits_{p>m}F^{ii}\frac{h^2_{1pi}}{\kappa_1-\kappa_p}\\
&+\frac{\bnt\phi'F^{-2}}{\kappa_1^2}F^{ii}h^2_{11i}-\lambda\bnt(\phi')^2F^{-2}\sum f^k.
\end{aligned}
\ee
Here and in the rest of this proof, we will use $C_1, C_2, \cdots$ and $c_0, c_1, c_2, \cdots$
to denote universal positive constants that only depend on $n, \rho, u,$ and $F.$
Equation \eqref{4.3.5} yields
\be\label{4.3.6}
\begin{aligned}
0&\leq\hat LG\leq\lambda C_1-\lt(\bnt\phi'F^{-2}+\lambda\bnt(\phi')^2F^{-2}\rt)\sum f^k\\
&-\bnt\phi'F^{-2}\sum f^k\kappa_k^2+2\bnt\phi'F^{-1}\kappa_1\\
&+\frac{\bnt\phi'}{\kappa_1F^2}\lt(-\frac{3}{2}F^{-1}F_1^2+F^{pq, rs}h_{pq1}h_{rs1}-2\sum\limits_i\sum\limits_{p>m}\frac{F^{ii}h^2_{1pi}}{\kappa_1-\kappa_p}
+\sum\limits_i\frac{F^{ii}h_{11i}^2}{\kappa_1}\rt).
\end{aligned}
\ee
By virtue of \eqref{0.4} we can see that
\begin{align*}
\s_2^{pq, rs}h_{pq1}h_{rs1}&=\sum\limits_{p\neq q}\s_2^{pp, qq}h_{pp1}h_{qq1}+2\sum\limits_{p>q}\s_2^{pq, qp}h^2_{pq1}\\
&=\sum\limits_{p\neq q}h_{pp1}h_{qq1}-2\sum\limits_{p>q}h^2_{pq1}\\
&\leq\sum\limits_{p\neq q}h_{pp1}h_{qq1}-2\sum\limits_{p>m}h^2_{11p}
\end{align*}
Note that $F^2=\s_2,$ thus we have
\[2FF^{pq, rs}h_{pq1}h_{rs1}+2(F_1)^2\leq\sum\limits_{p\neq q}h_{pp1}h_{qq1}-2\sum\limits_{p>m}h^2_{11p}.\]
This gives
\be\label{4.3.7}
F^{pq, rs}h_{pq1}h_{rs1}\leq\frac{1}{2}F^{-1}\lt(\sum\limits_{p\neq q}h_{pp1}h_{qq1}-2\sum\limits_{p>m}h^2_{11p}-2(F_1)^2\rt).
\ee
Therefore, we get
\begin{align*}
&-\frac{3}{2}F^{-1}F_1^2+F^{pq, rs}h_{pq1}h_{rs1}-2\sum\limits_{i}\sum\limits_{p>m}\frac{F^{ii}h^2_{1pi}}{\kappa_1-\kappa_p}+\sum\limits_i\frac{F^{ii}h_{11i}^2}{\kappa_1}\\
&\leq-\frac{5}{2}F^{-1}F_1^2+\frac{1}{2}F^{-1}\sum\limits_{p\neq q}h_{pp1}h_{qq1}-F^{-1}\sum\limits_{p>m}h^2_{11p}-2\sum\limits_{p>m}\frac{F^{11}h^2_{11p}}{\kappa_1-\kappa_p}\\
&-2\sum\limits_{p>m}\frac{F^{pp}h^2_{pp1}}{\kappa_1-\kappa_p}+\sum\limits_{p>m}\frac{F^{pp}h^2_{11p}}{\kappa_1}+\frac{F^{11}h^2_{111}}{\kappa_1},
\end{align*}
where we have used $h_{11i}=0$ for $1<i\leq m.$ By a straightforward calculation we can see that for each fixed $p>m,$
\begin{align*}
&\lt(-F^{-1}-2\frac{F^{11}}{\kappa_1-\kappa_p}+\frac{F^{pp}}{\kappa_1}\rt)h^2_{11p}\\
&=F^{-1}\lt(-1-\frac{\s_1-\kappa_1}{\kappa_1-\kappa_p}+\frac{\s_1-\kappa_p}{2\kappa_1}\rt)h^2_{11p}\\
&=F^{-1}\frac{(\kappa_1+\kappa_p)(\kappa_p-\s_1)}{2(\kappa_1-\kappa_p)\kappa_1}h^2_{11p}\leq 0,
\end{align*}
here we have used equality \eqref{0.3} and the first inequality in Lemma \ref{lem2.1.2}.
Thus we conclude
\be\label{4.3.8}
\begin{aligned}
&-\frac{3}{2}F^{-1}F_1^2+F^{pq, rs}h_{pq1}h_{rs1}-2\sum\limits_{i}\sum\limits_{p>m}\frac{F^{ii}h^2_{1pi}}{\kappa_1-\kappa_p}+\sum\limits_i\frac{F^{ii}h_{11i}^2}{\kappa_1}\\
&\leq\frac{1}{2}F^{-1}\sum\limits_{p\neq q}h_{pp1}h_{qq1}-\frac{5}{2}F^{-1}F_1^2-2\sum\limits_{p>m}\frac{F^{pp}h_{pp1}^2}{\kappa_1-\kappa_p}+\frac{F^{11}h^2_{111}}{\kappa_1}.
\end{aligned}
\ee
Now let $\epsilon>0$ be a small constant that will be determined later, $\delta=\delta_0=\frac{1}{2},$ and let $\delta'=\delta'(\epsilon, \delta, \delta_0)=O(\epsilon)>0$ be a constant determined by Lemma \ref{lem2.1.0}. We will divide the rest of this proof into two cases.

Case 1. When $\kappa_2\leq\delta'\kappa_1$ at $P_0,$ by Lemma \ref{lem2.1.0} we get
\be\label{4.3.9}
\sum\limits_{p\neq q}\frac{h_{pp1}h_{qq1}}{\s_2}-\frac{(\s_2^{ii}h_{ii1})^2}{\s_2^2}
\leq(\e-1)\frac{h^2_{111}}{\kappa_1^2}+\frac{1}{2}\sum\limits_{p>1}\frac{\s_2^{pp}h^2_{pp1}}{\kappa_1\s_2}.
\ee
Moreover, since $\kappa\in\Gamma_2$ we have $\sigma_2^{11}=\sum\limits_{i=2}^{n-1}\kappa_i+\kappa_n>0,$ which implies
\be\label{0.8}\kappa_n>-(n-2)\delta'\kappa_1.\ee
 When $\delta'>0$ small we obtain
\[2\sum\limits_{p>1}\frac{\s_2^{pp}h^2_{pp1}}{\kappa_1-\kappa_p}>2\sum\limits_{p>1}\frac{\s_2^{pp}h^2_{pp1}}{(1+(n-2)\delta')\kappa_1}
>\sum\limits_{p>1}\frac{\s_2^{pp}h^2_{pp1}}{\kappa_1}\]
Therefore, in this case \eqref{4.3.8} becomes
\be\label{4.3.9}
\begin{aligned}
&-\frac{3}{2}F^{-1}F_1^2+F^{pq, rs}h_{pq1}h_{rs1}-2\sum\limits_{i}\sum\limits_{p>1}\frac{F^{ii}h^2_{1pi}}{\kappa_1-\kappa_p}+\sum\limits_i\frac{F^{ii}h_{11i}^2}{\kappa_1}\\
&\leq\frac{1}{2}F^{-1}
\lt\{\sum\limits_{p\neq q}h_{pp1}h_{qq1}-\frac{5}{4}\s_2^{-1}(\s_2^{ii}h_{ii1})^2-2\sum\limits_{p>1}\frac{\s_2^{pp}h_{pp1}^2}{\kappa_1-\kappa_p}+\frac{\s_2^{11}h^2_{111}}{\kappa_1}\rt\}\\
&\leq\frac{1}{2}F^{-1}\lt\{(\e-1)\s_2\frac{h_{111}^2}{\kappa_1^2}+\frac{1}{2}\sum\limits_{p>1}\frac{\s_2^{pp}h^2_{pp1}}{\kappa_1}-\frac{1}{4}\s_2^{-1}(\s_2)^2_1
-2\sum\limits_{p>1}\frac{\s_2^{pp}h^2_{pp1}}{\kappa_1-\kappa_p}+\frac{\s_2^{11}h_{111}^2}{\kappa_1}\rt\}\\
&\leq\frac{1}{2}F^{-1}\lt[(\e-1)\s_2+\s_2^{11}\kappa_1\rt]\frac{h_{111}^2}{\kappa_1^2}
\end{aligned}
\ee
Plugging \eqref{4.3.9} into \eqref{4.3.6} and applying the first equality in Lemma \ref{lem2.1.1} we obtain
\be\label{4.3.10}
\begin{aligned}
0&\leq\hat LG\leq\lambda C_1-\lt(\bnt\phi'F^{-2}+\lambda\bnt(\phi')^2F^{-2}\rt)\sum f^k\\
&-\bnt\phi'F^{-2}\sum f^k\kappa_k^2+2\bnt\phi'F^{-1}\kappa_1+\frac{\bnt\phi'}{2F^3\kappa_1}[\e\s_2-\s_2(\kappa|1)]\frac{h^2_{111}}{\kappa_1^2}.
\end{aligned}
\ee
In view of our assumption that $\kappa_2\leq\delta'\kappa_1,$ we know
\[|\s_2(\kappa|1)|\leq c_0(\delta'\kappa_1)^2.\]
Therefore, we have
\be\label{4.3.11}
\begin{aligned}
0&\leq\hat LG\leq\lambda C_1-\lt(\bnt\phi'F^{-2}+\lambda\bnt(\phi')^2F^{-2}\rt)\sum f^k\\
&-\bnt\phi'F^{-2}\sum f^k\kappa_k^2+2\bnt\phi'F^{-1}\kappa_1+\frac{\bnt\phi'}{2F^3\kappa_1}[c_1\e+c_0(\delta'\kappa_1)^2]\frac{h^2_{111}}{\kappa_1^2}.
\end{aligned}
\ee
By \eqref{4.3.3} we get at $P_0$
\[\frac{h_{111}}{\kappa_1}=-\lambda\Phi_1=-\lambda\lt<V, \tau_1\rt>.\]
Combining with \eqref{4.3.11} and \eqref{0.5} yields
\be\label{4.3.12}
0\leq\hat LG\leq\lambda C_1+C_2(\delta')^2\lambda^2\kappa_1+C_3\kappa_1-\frac{n-1}{2}\lambda\bnt(\phi')^2F^{-3}\s_1.
\ee
Without loss of generality, we will always assume $\delta'\leq\frac{1}{2n^2},$ then by \eqref{0.8} we have, $$\s_1>\kappa_1+(n-1)\kappa_n>\frac{\kappa_1}{2}.$$
Therefore, \eqref{4.3.12} implies
\be\label{4.3.12*}
0\leq\hat LG\leq\lambda C_1+C_2(\delta')^2\lambda^2\kappa_1+C_3\kappa_1-\frac{n-1}{4}\lambda\bnt(\phi')^2F^{-3}\kappa_1.
\ee

We can choose $\lambda=\lambda(n, \rho, F, C_3)>0$ large such that $\frac{n-1}{4}\lambda\bnt(\phi')^2F^{-3}>2C_3+1$, then choose $\delta'=\delta'(\lambda, C_2)>0$ small (this can be achieved by choosing a small $\e$) such that $C_2(\delta')^2\lambda^2<\frac{1}{2}$. Then if at $P_0$ we have $\kappa_1>N_0=N_0(\lambda, C_1)>0$ large, the right hand side of \eqref{4.3.12*} would be strictly negative.
This leads to a contradiction.

Case 2. When $\kappa_2\geq\delta'\kappa_1$ at $P_0.$ By Lemma \ref{lem-F lower bound} and Lemma \ref{lem-F upper bound}
we know $f\lt(1, \frac{\kappa_2}{\kappa_1}, \frac{\kappa_3}{\kappa_1}, \cdots, \frac{\kappa_n}{\kappa_1}\rt)=O\lt(\frac{1}{\kappa_1}\rt).$
In view of our assumption that $\frac{\kappa_2}{\kappa_1}\geq\delta'$, when $\kappa_1$ is very large we have
$\kappa_n<0,$
and
\begin{align*}
\frac{C_4}{\kappa_1}&>f\lt(1, \frac{\kappa_2}{\kappa_1}, \frac{\kappa_3}{\kappa_1}, \cdots, \frac{\kappa_n}{\kappa_1}\rt)>\sigma_2\lt(1, \frac{\kappa_2}{\kappa_1}, \frac{\kappa_3}{\kappa_1}, \cdots, \frac{\kappa_n}{\kappa_1}\rt)\\
&=\sum\limits_{i\geq 2}\frac{\kappa_i}{\kappa_1}+\sigma_2\lt(0, \frac{\kappa_2}{\kappa_1}, \frac{\kappa_3}{\kappa_1}, \cdots,\frac{\kappa_n}{\kappa_1}\rt)\\
&>\delta'+c_2\frac{\kappa_n}{\kappa_1}
\end{align*}
for some $c_2=c_2(n)>0.$
Therefore, if $\kappa_1>N_1=N_1(1/\delta', C_4)$ very large at $P_0$ , then there exists $\eta_0=\eta_0(\delta', c_2)>0$ such that
$\frac{\kappa_n}{\kappa_1}\leq-\eta_0<0.$ This implies $$\sum f^k\kappa_k^2\geq f^n\kappa_n^2\geq\frac{1}{n}\lt(\sum f^i\rt)\eta_0^2\kappa_1^2\geq c_3\kappa_1^2,$$
where $c_3=c_3(n, \eta_0)>0.$ Combining with \eqref{4.3.6} we obtain
\be\label{4.3.13}
0\leq\hat LG\leq\lambda C_1-\bnt\phi'F^{-2}c_3\kappa_1^2+2\bnt\phi'F^{-1}\kappa_1+\frac{\bnt\phi'}{F^2}\sum F^{ii}\lambda^2\lt<V, \tau_i\rt>^2,
\ee
where we have used \eqref{4.3.3} and the fact that $F=\s_2^{1/2}$ is concave.
We can see that when $\kappa_1>N_2=N_2(N_1, n, \rho, u, F, \lambda, c_3)>0$ very large
\[\lambda C_1-\bnt\phi'F^{-2}c_3\kappa_1^2+2\bnt\phi'F^{-1}\kappa_1+\frac{\bnt\phi'}{F^2}\sum F^{ii}\lambda^2\lt<V, \tau_i\rt>^2<0.\]
This leads to a contradiction.
It follows that if $G$ achieves its maximum at an interior point, then $G$ is uniformly bounded. Therefore, the lemma is proved.
\end{proof}

So far we have obtained uniform $C^2$ bounds for the flow hypersurfaces $M_t.$ This implies the uniform parabolicity of the operator $L.$
Due to the concavity of the operator, we can apply the Krylov and Safonov regularity theorem (see \cite{Kry87}) to deduce $C^{2, \alpha}$ bounds. The $C^{\infty}$ bounds
follow from the Schauder theory. We conclude:

\begin{proposition}
\label{pro-longtime}Let $M_0\subset\mathbb S^{n+1}_1$ be a spacelike, compact, star-shaped, and strictly $2$-convex hypersurface. Then the flow \eqref{flow-equation}
exists for all time with uniform $C^\infty$-estimates.
\end{proposition}

\bigskip
\section{Convergence and inequality}
\label{sec5}
In this section, we complete the proof of the geometric inequality, that is, Corollary \ref{cor1.1}.
\begin{proof}[Proof of Corollary \ref{cor1.1}]
Recall Lemma \ref{lem3.1} we know that $\mathcal A_0$ is decreasing. By the $C^2$ estimates obtained in the subsection \ref{sub-section 4.3}, we have
$\mathcal A_2$ is uniformly bounded from above. Moreover, Lemma \ref{lem3.1} also tells us that $\mathcal A_2$ is increasing. Therefore, we have
\[\int_0^\infty\partial_t\mathcal A_2<\infty,\]
which implies that
\be\label{5.1}
\partial_t\mathcal A_2=3\int_{M_t}C_{n, 2}\phi'\s_2-\bnt\phi'\s_2^{-1/2}\s_3d\mu_g\goto 0
\ee
as $t\goto \infty.$ By Lemma \ref{lem-c0} we know
$\mathcal A_0=\int_{M_t}d\mu_g$ is bounded away from $0.$ Thus by \eqref{5.1}
we have $C_{n, 2}\phi'\s_2-\bnt\phi'\s_2^{-1/2}\s_3\goto 0$ as $t\goto\infty.$  In view of the Newton-Maclaurin inequality we conclude that any convergent subsequence $\{M_{t_i}\}$ must converge to a totally umbilical hypersurface, that is, a radial coordinate slice as $t_i\goto\infty.$ Note that by the proof of Lemma \ref{lem-c0} we can see that both $\rho_{\min}$ and $\rho_{\max}$ are monotonic. Therefore, we conclude that the limiting radial coordinate slice is unique.

In sum, we have the flow hypersurfaces $M_t$ converge to a radial coordinate slice smoothly as $t\goto\infty,$
and the inequality
\be\label{0.7} \int_M\s_2d\mu_g-(n-1)\mathcal A_0\leq\xi_{2, 0}(\mathcal A_0)\ee
follows from Lemma \ref{lem3.1} easily. Moreover, the equality holds if and only if $M$ is a radial coordinate slice. Following the argument in \cite{GL09},
when $M$ is $2$-convex instead of strictly $2$-convex, we may approximate it by strictly $2$-convex
star-shaped hypersurfaces. The inequality \eqref{0.7} follows from the approximation. Therefore, we complete the proof of Corollary \ref{cor1.1}.
\end{proof}

\end{document}